\documentclass[12pt,fleqn]{article}
\usepackage{geometry}
\usepackage{amsmath}
\usepackage{amssymb}
\usepackage{graphicx}
\usepackage{url}
\usepackage{xcolor}
\usepackage{epstopdf}

\usepackage{caption}
\newtheorem{theorem}{Theorem}[section]

\newtheorem{corollary}[theorem]{Corollary}
\newtheorem{proposition}[theorem]{Proposition}

\newtheorem{definition}[theorem]{Definition}

\newtheorem{remark}{Remark}

\newcommand{\fref}[1]{Figure~\ref{#1}}
\newcommand{\Fref}[1]{Figure~\ref{#1}}
\newcommand{\frefs}[1]{Figures~\ref{#1}}
\newcommand{\Frefs}[1]{Figures~\ref{#1}}
\newcommand{\sref}[1]{Section~\ref{#1}}

\newcommand{\atan}[1]{\tan^{-1}{\! #1}}
\newcommand{\billiard}[1][ABCD]{{\setlength{\fboxsep}{0.5mm} \setlength{\fboxrule}{0.4mm}{\color{blue} \framebox[12.5mm]{\color{black}{\sf #1}}}}}
\newcommand{\titlebilliard}{{\setlength{\fboxsep}{0.5mm} \setlength{\fboxrule}{0.4mm}{\color{blue} \framebox[10.5mm]{\color{black}{\sf ABCD}}}}}

\newenvironment{proof}{{\textbf{Proof.}}}{\mbox{} \\ \hspace*{\fill} $\square$ \\[5mm]}

\begin{document}

\title{Periodic orbits for square and rectangular billiards}

\author{Hongjia H.\ Chen and Hinke M.\ Osinga \\
\small Department of Mathematics, University of Auckland, Auckland, New Zealand}

\date{October 2024}

\maketitle

\begin{abstract} 
Mathematical billiards is much like the real game: a point mass, representing the ball, rolls in a straight line on a (perfectly  friction-less) table, striking the sides according to the law of reflection. A billiard trajectory is then completely characterised by the number of elastic collisions.  The rules of mathematical billiards may be simple, but the possible behaviours of billiard trajectories are endless. In fact, several fundamental theory questions in mathematics can be recast as billiards problems. A billiard trajectory is called a periodic orbit if the number of distinct collisions in the trajectory is finite. We classify all possible periodic orbits on square and rectangular tables. We show that periodic orbits on such billiard tables cannot have an odd number of distinct collisions. We also present a connection between the number of different classes of periodic orbits and Euler's totient function, which for any integer $N$ counts how many integers smaller than $N$ share no common divisor with $N$ other than $1$. We explore how to construct periodic orbits with a prescribed (even) number of distinct collisions, and investigate properties of inadmissible (singular) trajectories, which are trajectories that eventually terminate at a vertex (a table corner).
\end{abstract}

\begin{quote}
\end{quote}

\section{Introduction}
Mathematical billiards shares many similarities to the game of billiards in reality. Fundamentally, both comprise a billiard ball moving in a straight line until it strikes a side of the table. Mathematical billiards ignores friction, takes the billiard ball as a point mass and assumes purely elastic collisions. As a consequence, the billiard trajectory will satisfy the law of reflection (\emph{the angle of incidence equals the angle of reflection}) when it collides with a boundary. Therefore, any billiard trajectory is uniquely determined by its initial position and direction of motion. A real billiard table is almost always rectangular, but mathematical billiard tables can have arbitrary shapes and dimensions. First posed as Alhazen's problem in optics by Ptolemy in 150 AD~\cite{baker1881}, mathematical billiards is much more than a fun and interesting game; it turns out that the mere shape of the table distinguishes mathematical billiards into three different classes---elliptic, hyperbolic, and parabolic---which are well known as classes in different fields of mathematics, not the least in dynamical systems theory~\cite{katok2009, sinai1991}. Mathematical billiards has extensively been studied from the perspective of ergodic theory in dynamical systems, as well as algebraic geometry (moduli spaces) and Teichm\"{u}ller theory; for example, see the textbooks by Tabachnikov~\cite{sergebook}, Chernov and Markarian~\cite{chernovbook}, or Rozikov~\cite{rozikovbook}. Nevertheless, a myriad of unsolved open problems in mathematical billiards remain~\cite{OpenProblems2022, dettmann2011, gutkin2012, KaloshinSorrentino2022}. Indeed, there are many fundamental theory questions in mathematics that can be recast as a billiard problem~\cite{sinai2000, sinai2004}, as well as a plethora of applications, including but not limited to the computation of $\pi$~\cite{galperin2003}, mechanics~\cite{Arnold1989}, quantum computing~\cite{quantum1982}, pouring problems~\cite{movies2012}, Benford's Law~\cite{sergebook}, diffusion in the Lorentz Gas~\cite{dettmann2014} and the Riemann hypothesis~\cite{dettmann2005}.

For the class of planar, polygonal tables, that is, a flat shape bounded by a piecewise-linear closed curve, any billiard trajectory exhibits one of only three possible behaviours: 
\begin{enumerate}
\item it is a \emph{singular orbit}---after a finite number of collisions, the billiard trajectory terminates at one of the vertices of the table;
\item it is a \emph{periodic orbit}---after a finite number of collisions, the billiard trajectory retraces itself;
\item it is a \emph{non-periodic orbit}---the billiard trajectory continues indefinitely, without ever repeating. 
\end{enumerate}
Obviously, singular orbits only exist on tables with vertices and they are often ignored or classified as non-periodic~\cite{singular_traj_ref}. If a billiard trajectory is periodic, its period is given by the total number of unique collisions. A natural question to ask is whether periodic orbits always exist or whether this depends on the shape of the billiard table. Furthermore, what periods are possible? In fact, the very question whether periodic orbits exist on any polygon is still an open problem! It is listed as Problem 3 in Katok's five most resistant problems in dynamics~\cite{katok}. Even the case of a triangular billiard table is considered impenetrable, although the Triangular Billiards Conjecture, which states that every triangular table has a periodic billiard path, is widely believed to be true~\cite{schwartz2009}. 

The question about which periods are possible has been treated for specific polygonal billiard tables. In particular, Baxter and Umble~\cite{baxter} classified all possible periodic orbits for an equilateral triangle. In this paper, we consider square or rectangular billiard tables and we similarly give a complete classification of which periods and types of periodic orbits are possible for a square or rectangle. The square is perhaps the simplest piecewise-smooth billiard table, which means that it is introduced early on in billiard textbooks, typically alongside the smooth, circular billiard; for example, see~\cite{chernovbook, rozikovbook, sergebook}. As a consequence, results for the square billiard tend to be skimmed over, despite its behaviour being distinctly different from circular and other smooth, convex tables. Existence of periodic orbits on the square and rectangular billiard is easily verified by explicit examples. \Fref{fig:square_examples} shows three examples of periodic orbits on the square. Here, the square table is represented by the (blue) boundary with vertices {\sf A}, {\sf B}, {\sf C} and {\sf D}, the periodic orbits are indicated as directed (black) lines that collide with the boundary at points identified by thick (red) dots. By counting these points of collision, we see that the periods are two, four and six in the respective panels~(a), (b) and~(c). Note that the period-two orbit in panel~(a) is readily produced on a rectangle as well; for the period-four and -six orbits this is less obvious.
%
\begin{figure}[t!]
  \centering
  \includegraphics{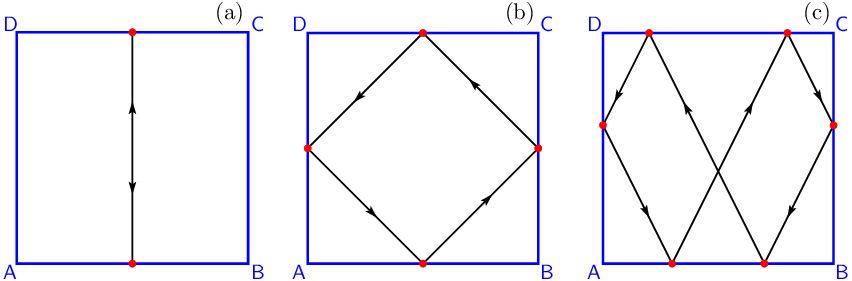}
  \caption{\label{fig:square_examples}
    Examples of periodic orbits on the square billiard table; the orbits in panels~(a), (b) and~(c) have periods two, four and six, respectively.}
\end{figure}

Interestingly, none of the textbooks show that the square and rectangular billiard tables do not admit periodic orbits with odd periods. This can be proven using cutting sequences, as done by Davis~\cite{davissquare}, but we provide a different proof that harmonises well with the periodicity classification presented in this paper. Furthermore, we also discuss how to construct periodic orbits with prespecified periods and connect our classification of periodic orbits with Euler's totient function from number theory~\cite{totient}. 

This paper is organised as follows. In the next section, we define the different classes of periodic orbits and clarify what we mean when two periodic orbits with the same period are different. Here, we also explain the powerful technique of unfolding~\cite{FoxKershner1936, KatokZemlyakov1975} that is used in the classical proof of the existence of periodic orbits on the square or rectangular billiard table. We present alternative and intuitive proofs for periodicity and non-existence of odd periodic orbits in~\sref{sec:square}. In \sref{sec:classify}, we count and fully classify all possible periodic orbits. We treat singular orbits in \sref{sec:gendiag}, where we show how different types of singular orbits form boundaries between different families of periodic orbits. In \sref{sec:rectangle}, we extend the results for the square to the rectangular billiard table. We conclude in \sref{sec:conclusion} with a discussion of future work.

\section{Setting and definitions}
\label{sec:setting}
The three periodic orbits shown in \fref{fig:square_examples} for the square billiard table are all different, because their periods are not the same. On the other hand, the period-two orbit in \fref{fig:square_examples}(a), for example, can be shifted to the right or left without changing its period; similarly, the period-six orbit in \fref{fig:square_examples}(c) can be flipped upside down without changing the period. In this section, we define the \emph{family} or \emph{class} of periodic orbits that we consider to be equivalent, and we will then proceed to count the number of different classes of periodic orbits.

We begin by defining a coordinate system that identifies a quadrilateral billiard table, denoted $\billiard$, with four (ordered) vertices {\sf A}, {\sf B}, {\sf C} and {\sf D}. Note that the billiard table can (uniformly) be scaled without affecting the number of distinct collisions for a periodic orbit; this means that we may assume that side {\sf AB} has unit length. Moreover, we consider this side as the base and position the table such that {\sf AB} is equal to the unit interval $\mbox{\sf AB} := \left\{ (x, 0) \in \mathbb{R}^2 \bigm| 0 \leq x \leq 1 \right\}$ on the $x$-axis, and the vertex {\sf A} lies at the origin. Recall that any billiard trajectory is uniquely determined by an initial position and associated direction of motion. We assume that the initial position is a point of collision with the side {\sf AB}; if necessary, we rotate the table and relabel the vertices so that such a point exists. Then we can identify any billiard trajectory by the initial point of collision on the side {\sf AB} at distance $P_0 \in [0, 1)$ from {\sf A}, and an initial angle $\alpha_{0} \in (0, \frac{\pi}{2}]$ between the corresponding outgoing (or equivalently incoming) line and the side {\sf AB}; we will say that the pair $\langle P_0,  \alpha_0 \rangle$ generates the billiard trajectory. Note that we restrict the angle $\alpha_0$ to be at most $\frac{\pi}{2}$ radians, which means that we do not specify the direction of motion. Since we are primarily interested in periodic orbits, and how many different ones there are, we do not care whether the sequence of collision points is given in reversed order. We will slightly abuse notation and refer to $P_0$ as both the initial point on {\sf AB} and its coordinate on the $x$-axis.

We only consider square or rectangular billiard tables; hence, the adjacent sides are perpendicular, which has an important consequence:
\begin{proposition}
\label{prop:angles}
Consider a billiard trajectory on the square or rectangular billiard $\billiard$ generated by the pair $\langle P_0,  \alpha_0 \rangle$ with $P_0 \in [0, 1)$ and $\alpha_0 \in (0, \frac{\pi}{2}]$. Then all collisions with sides {\sf AB} and {\sf CD} will be at angle $\alpha_0$ and those with sides {\sf BC} and {\sf DA} will be at angle $\frac{\pi}{2} - \alpha_0$.
\end{proposition}
\begin{proof}
Since $\alpha_0 \in (0, \frac{\pi}{2}]$, the outgoing line from $P_0$ on the side {\sf AB} will either end on the adjacent side {\sf BC} or on the opposite side {\sf CD}. The sides {\sf AB} and {\sf CD} are parallel, because $\billiard$ is either a square or a rectangle. Hence, if the outgoing line from $P_0$ ends on {\sf CD}, is does so at the same angle $\alpha_0$; see \fref{fig:square_examples}(a) and~\ref{fig:square_examples}(c) for examples. If the outgoing line from $P_0$ ends on the adjacent side {\sf BC} then it does so at an angle $\alpha_1$ such that $\alpha_0 + \frac{\pi}{2} + \alpha_1 = \pi$, because the outgoing line forms a right triangle with the sides {\sf AB} and {\sf BC}; see \fref{fig:square_examples}(b) and~\ref{fig:square_examples}(c) for examples. Therefore, we have $\alpha_1 = \frac{\pi}{2} - \alpha_0$ as claimed. The end point of the outgoing line from $P_0$ is another collision point $P_1$, unless it is a vertex, in which case the billiard trajectory terminates and the result holds. By rotating the table and relabelling the vertices, we can consider the billiard trajectory as being generated by the pair $\langle P_1, \alpha_1 \rangle$, with $\alpha_1 = \alpha_0$ or $\alpha_1 = \frac{\pi}{2} - \alpha_0$. Using the same arguments, we find that a collision with the (original) side {\sf DA} also occurs at angle $\frac{\pi}{2} - \alpha_0$. It follows that all collisions with sides {\sf AB} and {\sf CD} occur at angle $\alpha_0$ and all collisions with sides {\sf BC} and {\sf DA} occur at angle $\frac{\pi}{2} - \alpha_0$. 
\end{proof}
We use Proposition~\ref{prop:angles} to classify different billiard trajectories with the same period: we distinguish them by the numbers of different collision points in the cycle that correspond to the two different angles at which these collisions occur. More precisely, we have the following definition of equivalence.
\begin{definition}[Equivalence class $\mathcal{C}_K(p)$ of period-$K$ orbits]
\label{def:POclass}
A period-$K$ orbit for the square (or rectangular) billiard $\billiard$ belongs to the equivalence class $\mathcal{C}_K(p)$, for some integer $p < K$, if a total of $p$ of its $K$ different collision points lie on the (parallel) sides {\sf AB} or {\sf CD}.
\end{definition}
Note that any period-$K$ orbit from the class $\mathcal{C}_K(p)$ has $q = K - p$ different collision points on the (parallel) sides {\sf BC} or {\sf DA}, which we may emphasise by saying the period-$K$ orbit is of type $(p, q)$. Definition~\ref{def:POclass} makes no distinction between periodic orbits that have the exactly same number of collisions on each of the horizontal and vertical sides.

\begin{figure}[t!]
  \centering
  \includegraphics{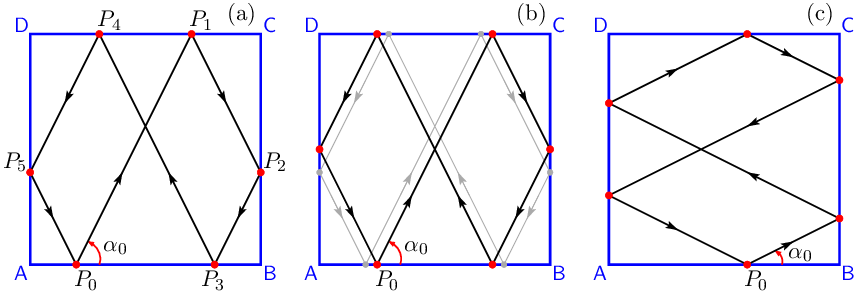}
  \caption{\label{fig:equivalence}
    Examples of period-six orbits on the square billiard, generated by the initial pair $\langle P_0, \alpha_0 \rangle = \langle 0.2,\, \atan{(2)} \rangle$ in panel~(a), the shifted initial pair $\langle P_0, \alpha_0 \rangle = \langle 0.25,\, \atan{(2)} \rangle$ in panel~(b) that is from the same equivalence class, and the pair $\langle P_0, \alpha_0 \rangle = \langle 0.6,\, \frac{\pi}{2} - \atan{(2)} \rangle$ obtained by rotation in panel~(c), which we consider part of a different family of period-six orbits.}
\end{figure}
%
The three different periodic orbits shown in \fref{fig:square_examples} with periods two, four, and six, can be distinguished more precisely as periodic orbits from the three different equivalence classes $\mathcal{C}_2(2)$, $\mathcal{C}_4(2)$ and $\mathcal{C}_6(4)$, respectively. As a more detailed example, compare the period-six orbits in \fref{fig:equivalence} with the possibly equivalent period-six orbit from \fref{fig:square_examples}(c). We identify collision points by their coordinates on the boundary of the unit square, that is, we assume vertices {\sf A} and {\sf B} have coordinates $(0, 0)$ and $(1, 0)$, and by definition, vertices {\sf C} and {\sf D} then have coordinates $(1, 1)$ and $(0, 1)$, respectively. Panel~(a) shows the period-six orbit of \fref{fig:square_examples}(c), but reflected about the line $\{ y = 0.5 \}$. More precisely, this period-six orbit is generated by the pair $\langle P_0, \alpha_0 \rangle = \langle 0.2,\, \atan{(2)} \rangle$, which means that it starts at the point $P_0 = (0.2, 0)$ on the side {\sf AB} along the outgoing line with slope $2$. The sequence of successive collision points is then given by the points $P_1 = (0.7, 1)$, $P_2 = (1, 0.4)$, $P_3 = (0.8, 0)$, $P_4 = (0.3, 1)$, $P_5 = (0, 0.4)$, after which the cycle repeats with $P_0$. Note that all lines are parallel to either the outgoing line at $P_0$, with slope $2$, or the incoming line at $P_0$, with slope $-2$. Hence, the angles at $P_0$ and $P_3$ on the horizontal side {\sf AB} are both $\alpha_0$, as are the angles at $P_1$ and $P_4$ on the other horizontal side {\sf CD}; on the other hand, with respect to the verical sides  {\sf BC} and {\sf DA}, these lines only have slopes $\pm \frac{1}{2}$ so that the angles at $P_2$ and $P_5$ are $\atan{(\frac{1}{2})} = \frac{\pi}{2} - \alpha_0$.  Therefore, this period-six orbit is of type $(4, 2)$ and it is equivalent to the period-six orbit from \fref{fig:square_examples}(c).  If we choose the reversed direction of motion, starting along the line with slope $-2$, we encounter the reversed sequence of collisions $P_5$, $P_4$, $P_3$, $P_2$, $P_1$, and $P_0$, leading again to an equivalent periodic orbit of type $(4, 2)$. Similarly, we may shift the point $P_0$, say, to the point $(0.25, 0)$ as shown in panel~(b), while maintaining the same angle $\alpha_0$; the associated periodic orbit is again of type $(4,2)$. A further shift of $P_0$ to $(0.3, 0)$ leads to the period-six orbit from \fref{fig:square_examples}(c). Indeed, these three periodic orbits are all part of a family of period-six orbits in $C_6(4)$ that is generated by a pair $\langle P_0, \alpha_0 \rangle$, with $P_0$ varying over one or more sub-intervals in $[0, 1]$.

\Fref{fig:equivalence}(c) shows the same period-six orbit from panel~(a), but rotated anti-clockwise by a quarter turn; a mere relabelling of the vertices leads to a periodic orbit generated by the pair $\langle P_5, \atan{(\frac{1}{2})} \rangle = \langle 0.6,\, \atan{(\frac{1}{2})} \rangle$, which is an example from a different equivalence class, namely, from $\mathcal{C}_6(2)$. We make a distinction between these two families of period-six orbits, because it would be natural to do so for rectangular billiard tables. Therefore, in this paper, the periodic orbit of type $(2, 4)$ shown in \fref{fig:equivalence}(c) is not from the same equivalence class as the periodic orbits of type $(4, 2)$ shown in panels~(a) and~(b).

\subsection{Unfolding of billiard trajectories}
\label{sec:square_unfolding}
Following a billiard trajectory on a table $\billiard$ can be difficult if there are many collisions, because of a myriad of intersecting lines. The technique of \emph{unfolding} transforms the billiard trajectory: rather than reflecting collisions with the sides, the entire table reflects, so that the billiard trajectory remains a straight line. For so-called \emph{rational}, planar, polygonal billiard tables, which have polygonal angles that are rational multiples of $\pi$, this approach relates mathematical billards to the theory of geodesic flow on Riemann surfaces~\cite{FoxKershner1936, KatokZemlyakov1975, keane1978}; in special cases of rational billard tables, including the square and rectangle, this unfolding leads to a tiling of the plane and is amenable to illustration. 

\begin{figure}[t!]
  \centering
  \includegraphics{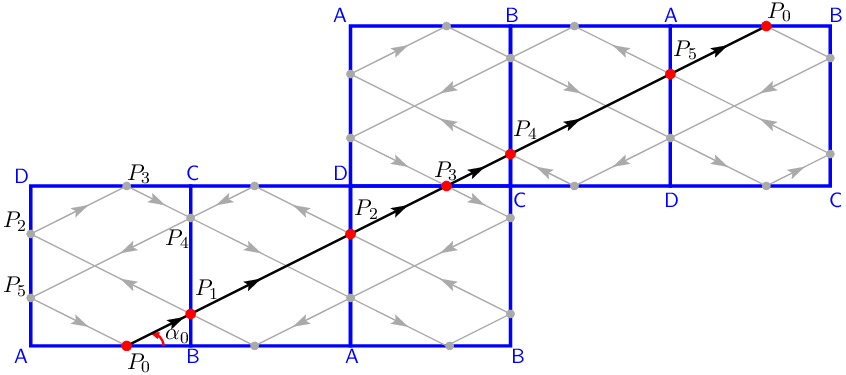}
  \caption{\label{fig:square_unfolding} 
    The unfolded period-six orbit from \fref{fig:equivalence}(c) on the square $\titlebilliard$. The periodic orbit has type $(2, 4)$, which means that the unfolding requires two reflections about a horizontal side and four reflections about a vertical side before the trajectory repeats on a translated copy of tables with the same orientations.}
\end{figure}
%
We show an example in \fref{fig:square_unfolding} with an unfolding of the period-six orbit from \fref{fig:equivalence}(c). Recall that this periodic orbit is generated by the pair $\langle P_0, \alpha_0 \rangle = \langle 0.6,\, \atan{(\frac{1}{2})} \rangle$ on the square $\billiard$ and there are six different collisions, which are the points $P_0 = (0.6, 0)$ on the side {\sf AB}, followed by $P_1 = (1, 0.2)$, $P_2 = (0, 0.7)$, $P_3 = (0.6, 1)$, $P_4 = (1, 0.8)$, and $P_5 = (0, 0.3)$. Here, $P_1$ and $P_4$ lie on the side {\sf BC}, and $P_2$ and $P_5$ lie on the side {\sf DA}, while $P_3$ lies on {\sf CD}. In the unfolding, the trajectory from $P_0$ to $P_1$ continues in a straight line with slope $\frac{1}{2}$ beyond the side {\sf BC}, on the (horizontally) reflected table $\billiard[BADC]$; the next collision occurs on the opposite side, which is {\sf AD} as required, at the same location $P_2$ as would have been reached by following the outgoing line with slope $-\frac{1}{2}$ from $P_1$ in the original table $\billiard$. At $P_2$, the table is again reflected horizontally and the trajectory continues in the original orientation until the collision at $P_3$ on the side {\sf CD}. Next, the trajectory continues with the same slope beyond {\sf CD}, on the vertically reflected table $\billiard[DCBA]$; the subsequent collision points $P_4$ and $P_5$, together with two associated vertical reflections, are oriented up-side-down with respect to the original table $\billiard$. Indeed, the trajectory from $P_4$ to $P_5$ continues at the same angle $\atan{(\frac{1}{2})}$ with the horizontal, because in the unfolding, it lies on the table $\billiard[CDAB]$. The horizontal reflection used to continue past $P_5$ flips the table back to the orientation $\billiard[DCBA]$ that was used to pass from $P_3$ to $P_4$. The next collision with the side {\sf AB} at the top right in \fref{fig:square_unfolding} corresponds to the point $P_0$ on {\sf AB}; the table obtained after a sixth reflection about the side {\sf AB} will be a translated copy of the original table $\billiard$ on the bottom left, and the trajectory will repeat. Hence, the line segment in \fref{fig:square_unfolding} can be continued past $P_0$ in both directions to give a straight line in $\mathbb{R}^2$ that has slope $\frac{1}{2}$ and passes through the point $P_0$ on the $x$-axis; the entire line is the unfolded trajectory in the plane, which in this example represents a period-six orbit of type $(2,4)$.
%
\begin{remark}
\label{rem:unfolding}
The technique of unfolding for the rectangular billiard is done in the exact same way with horizontal and vertical reflections. Since the direction of motion at a point of collision only depends on the angle it makes with that side of the table, trajectories on a rectangular billiard also correspond to straight lines in $\mathbb{R}^2$ determined by the corresponding generator $\langle P_0, \alpha_0 \rangle $.
\end{remark}

\section{Properties of the square billiard}
\label{sec:square}
Notice in \fref{fig:square_unfolding} that the orientation of the table changes each time a collision occurs. More precisely, there are four unique orientations, namely, the original table $\billiard$, the horizontally reflected version $\billiard[BADC]$, the vertically reflected version $\billiard[DCBA]$, and the table $\billiard[CDAB]$ that is both horizontally and vertically reflected with respect to the original table. We can identify these four orientations in terms of the integer coordinates of their vertices in $\mathbb{R}$:
%
\begin{definition}
\label{def:unfolding_orientation}
Consider a square tile in $\mathbb{R}^{2}$ with bottom-left vertex $(i,j) \in \mathbb{Z} \times \mathbb{Z}$. Then this tile corresponds to only one of the following four orientations:
\begin{itemize}
\item if both $i$ and $j$ are even then the tile corresponds to $\billiard$, which has positive horizontal and positive vertical orientation;
\item if $i$ is odd and $j$ is even then the tile corresponds to $\billiard[BADC]$, which has negative horizontal and positive vertical orientation;
\item if both $i$ and $j$ are odd then the tile corresponds to $\billiard[CDAB]$, which has negative horizontal and negative vertical orientation;
\item if $i$ is even and $j$ is odd then the tile corresponds to $\billiard[DCBA]$, which has positive horizontal and negative vertical orientation.
\end{itemize}
\end{definition}
%
\begin{figure}[t!]
  \centering
  \includegraphics{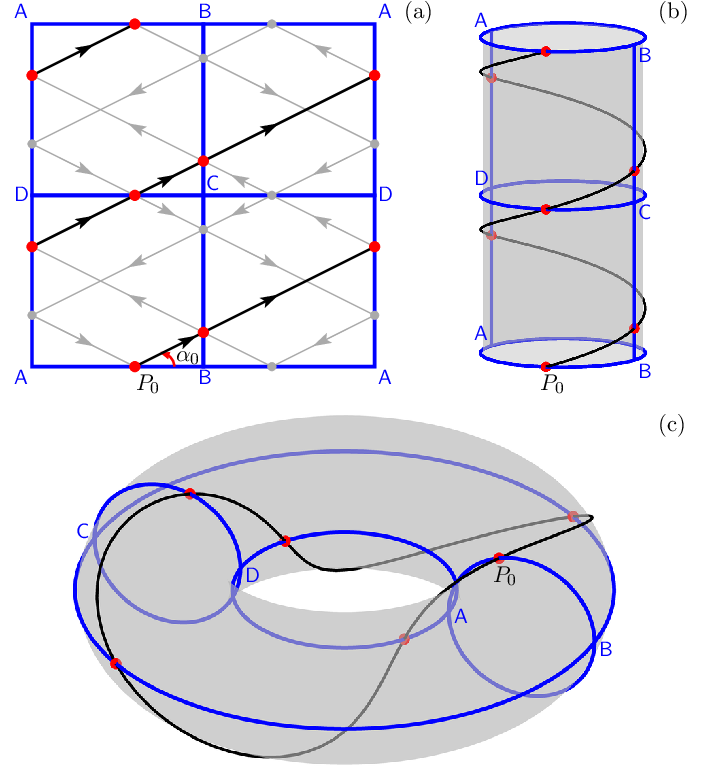}
  \caption{\label{fig:torus}
    Equivalent representation of the period-six orbit from \fref{fig:equivalence}(c) on the torus. Panel (a) illustrates the trajectory on the large square table comprising all four orientations of the square $\titlebilliard$. The left and right sides {\sf ADA} are identified to form a cylinder in panel~(b), and the top and bottom sides {\sf ABA} are subsequently identified to form the torus in panel~(c), on which the period-six orbit forms a closed curve.}
\end{figure}
%
By combining all four orientations together, we can also view $\mathbb{R}^{2}$ as being tiled with squares that have sides {\sf ADA} and {\sf ABA}, twice the length of those for the original table $\billiard$. The tiling with such larger squares is done by mere translations, rather than reflections. Consequently, this larger square is a representation of the (flat) torus $\mathbb{T}^2$, as illustrated in \fref{fig:torus}. Here, panel~(a) shows again the unfolding of the period-six orbit from \fref{fig:equivalence}(c), but as this billard trajectory disappears off the edge of the square on the right-hand side {\sf ADA}, it reappears on the left-hand side {\sf ADA} and continues with the same slope; similarly, when the top side {\sf ABA} is reached, the trajectory reappears on the bottom side {\sf ABA}. Hence, the period-six orbit is now represented by a set of three parallal straight-line segments, rather than a single straight line. These line segments form a single curve if we identify the left- and right-hand sides {\sf ADA} by folding the square into a cylinder; see \fref{fig:torus}(b). If we now also identify the top and bottom sides {\sf ABA}, as done in \fref{fig:torus}(c), then the period-six orbit is, in fact, given by a closed curve on $\mathbb{T}^2$.

The alternative representation in \fref{fig:torus} is the basis for the classical proof of existence and classification of periodic orbits on the square billiard, which uses the notion of \emph{geodesic} on $\mathbb{T}^2$, that is, a length-minimising curve. In the usual Euclidean geometry, geodesics are straight lines, and $\mathbb{T}^2$ inherits this geometry by way of the construction illustrated in \fref{fig:torus}. The theory of geodesics tell us that all  straight lines with rational slopes correspond to closed geodesics on $\mathbb{T}^2$; for example, see~\cite{petersen}. Therefore, if a trajectory on the square billiard is periodic, then the slope must be rational.
%
\begin{remark}
\label{rem:dichotomy}
The theory of geodesics also states that all straight lines with irrational slopes correspond to geodesics that densely fill the torus. This dichotomy is known as \emph{Veech dichotomy} in the literature; it implies that the square billiard is \emph{ergodically optimal}~\cite{veech1989}. We refer again to~\cite{petersen} for details. 
\end{remark}
%
The classical proof has the advantage that it extends readily to higher dimensions by using analogous arguments~\cite{sergebook}. However, this approach ignores the possibility that a billiard trajectory with rational slope terminates at one of the four vertices; indeed, the condition that the slope be rational is not sufficient to guarantee periodicity. In this paper, we analyse the \emph{trichotomy} that includes the possibility of singular orbits for a square billiard. We use the technique of unfolding in $\mathbb{R}^2$ as an alternative and more intuitive way to prove important properties of the square billiard, such as the non-existence of periodic orbits with odd periods~\cite{davissquare} and what the period of a billiard trajectory will be for a given rational slope. We will also count and classify all different equivalence classes of periodic orbits that have a given period associated with a particular rational slope.

\subsection{Periods of periodic orbits for the square billiard are even}
\label{square_odd_section}
All the examples of periodic orbits for the square billiard given so far have even periods. This is not a coincidence. In particular, is it not possible to have a periodic orbit with period three.
%
\begin{figure}[t!]
  \centering
  \includegraphics{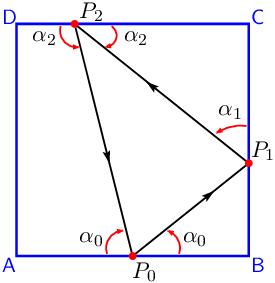}
  \caption{\label{fig:square_period3}
    Hypothetical period-three orbit on the square $\titlebilliard$ shown as a triangle composed of the three points of collision $P_0$, $P_1$, and $P_2$ at which the billiard trajectory makes angles $\alpha_{0}$, $\alpha_{1}$ and $\alpha_{2}$ with the table, respectively. The incoming and outgoing angles at these points are intended to be equal, and labelled so accordingly.} 
\end{figure}
%
\begin{proposition}
\label{prop:square_period3}
The square billiard does not admit a period-three orbit.
\end{proposition}
\begin{proof}
Suppose for the sake of contradiction that a period-three orbit exists. To focus the mind, \fref{fig:square_period3} illustrates this hypothetical billiard trajectory on the square $\billiard$. The three different points of collision are denoted $P_0$, $P_1$ and $P_2$, and their associated angles are $\alpha_{0}$, $\alpha_{1}$ and $\alpha_{2}$, respectively; here, angles with the same label are supposed to be equal, even though the image may suggest otherwise. Note that any period-three orbit must collide with three different sides of the square, because it is impossible for a billiard trajectory to incur two consecutive collisions on the same side. Since we can always rotate the table and relabel the vertices, the illustration in \fref{fig:square_period3} is representative for any period-three orbit. In other words, without loss of generality, we may assume that $P_0$ lies on the side {\sf AB}, and the other two, successive points of collision $P_1$ and $P_2$ lie on the sides {\sf BC} and {\sf CD}, respectively. Recall from \sref{sec:setting} that the billiard trajectory is uniquely generated by the pair $\langle P_0, \alpha_0 \rangle$, and Proposition~\ref{prop:angles} implies that we must have $\alpha_1 = \frac{\pi}{2} - \alpha_0$ and $\alpha_2 = \alpha_0$. Furthermore, the sum of the angles in the quadrilateral formed by $P_0$, $P_2$, and the two vertices {\sf D} and {\sf A} of the square table must be $2 \pi$, which implies that $\alpha_0 + \alpha_2 = \pi$. Therefore, if a period-three orbit exists then
\begin{displaymath}
  \left. \begin{array}{rcrcl}
            \alpha_0 &+& \alpha_2 &=& \pi \\
                            & & \alpha_2 &=& \alpha_0
          \end{array} \right\} \implies 2 \, \alpha_0 = \pi \Longleftrightarrow \alpha_0 = \tfrac{\pi}{2}.
\end{displaymath}
However, the pair $\langle P_0, \frac{\pi}{2} \rangle$ generates a period-two orbit for all $P_0 \in (0, 1)$, because the billiard trajectory will bounce back and forth between the two points $P_0$ and $P_2$; a contradiction.
\end{proof}
We cannot easily extend this geometric proof to other odd periods as the number of different cases grows very quickly and self-intersections cause difficulty to make the same arguments. Therefore, we use the technique of unfolding, explained in \sref{sec:square_unfolding}, to prove the following general result.
%
\begin{theorem}
\label{thm:square_periodic}
Any periodic orbit for the square billiard is generated by a pair $\langle P_0, \alpha_0 \rangle$, with $\alpha_0 \in (0, \frac{\pi}{2}]$ such that $\tan{(\alpha_0)}$ is rational, and the period of this periodic orbit is even. 
\end{theorem}
%
\begin{proof}
Consider a periodic orbit for the square $\billiard$ generated by a pair $\langle P_0, \alpha_0 \rangle$ with $P_0 \in [0, 1)$ and $\alpha \in (0, \frac{\pi}{2}]$. Note that $P_0 \neq 0$, because a periodic billiard trajectory that starts at a vertex must also return to this vertex, which would make it a singular orbit. The unfolding on $\mathbb{R}^{2}$ of such a periodic orbit is a straight line with slope $\tan{(\alpha_0)}$ that passes through the point at distance $P_0$ from the vertex {\sf A} at the origin, and such that it intersects the tiling of $\mathbb{R}^2$ along other (although not necessarily all) sides {\sf AB} at distances $P_0$ from the shifted vertex {\sf A} in translated copies of the original table $\billiard$.
%
\begin{figure}[t!]
  \centering
  \includegraphics{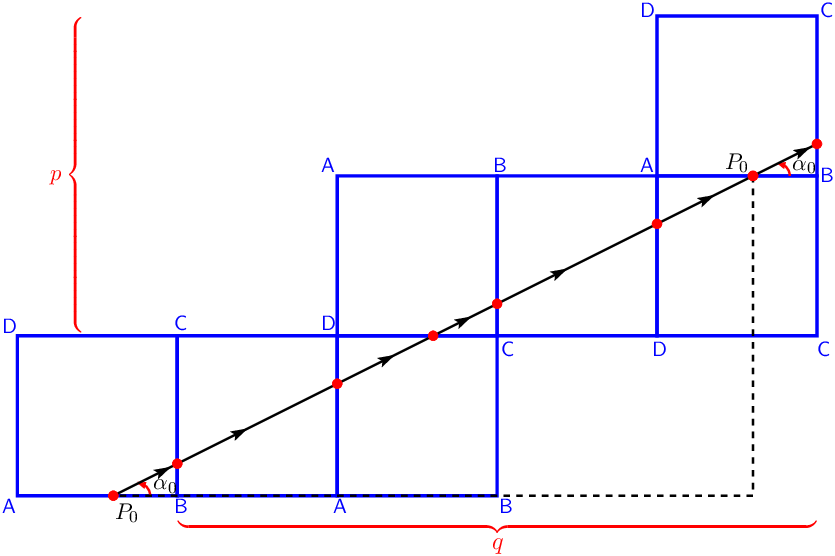}
  \caption{\label{fig:square_alternate_periodic}
    Unfolding on $\mathbb{R}^{2}$ of a periodic orbit for the square $\titlebilliard$ with period $p + q$ composed of $p$ vertical and $q$ horizontal reflections before returning to the initial position $P_0$.} 
\end{figure}
%
An example is shown in \fref{fig:square_alternate_periodic}, where again the period-six orbit from \fref{fig:equivalence}(c) was used. The slope $\tan{(\alpha_0)}$ of the unfolded trajectory is determined by the difference between the $x$- and $y$-coordinates of the initial point $P_0$ and its translated copy; this is suggestively indicated by the dashed triangle in \fref{fig:square_alternate_periodic}. Suppose that the periodic orbit is of type $(p, q)$, for some $p$, $q \in \mathbb{N}$, which means that the unfolded trajectory passes through a translated copy of $P_0$ (with the correct orientation) for the first time after $p$ vertical and $q$ horizontal reflections. Since the distance between $P_0$ and the vertex {\sf A} is preserved in the translated copy, the line through $(0,0)$ and $(q, p)$ in $\mathbb{R}^2$ is parallel to the unfolded trajectory. Hence, the slope of the unfolded trajectory is $\tan{(\alpha_0)} = \frac{p}{q}$, which is rational, as claimed.

Furthermore,  Definition~\ref{def:unfolding_orientation} states that a translated copy of the billiard table with orientation $\billiard$ is only obtained when the (integer) coordinates of the shifted vertex {\sf A} are both even, that is, $p$ and $q$ must both be even. The period of the periodic orbit is given by the total numer of distinct collisions; we claim that this number is $p + q$, and since $p$ and $q$ are even, the period of the periodic orbit is even. It remains to prove that the total numer, say $K$, of distinct collisions is $p + q$. Recall that the technique of unfolding transforms the billiard trajectory into a straight line by reflecting the table at points of collision, instead of changing the direction of motion. Hence, $K$ is at most equal to the number of reflections, that is, $K \leq p + q$. Now suppose $K < p + q$, which means that the billiard trajectory is repeated after $K$ collisions. However, $p$ and $q$ are, by definition, the number of horizontal and vertical reflections needed to pass through a translated copy of $P_0$ with the correct orientation for the \emph{first} time. Hence, $P_\ell \neq P_0$ for all $0 < \ell < p + q$; therefore, $K = p + q$ and the period is even.
\end{proof}
In what follows, we characterise all periodic orbits of period $K$, for any integer $K \in 2 \, \mathbb{N}$, in terms of the conditions we must impose on $P_0 \in [0, 1)$ and $\alpha \in (0, \frac{\pi}{2}]$ such that the pair $\langle P_0, \alpha_0 \rangle$ gives rise to such a period-$K$ orbit for the square billiard $\billiard$. In particular, we confirm that periodic orbits exist for all even periods, and explain how many different equivalence classes there are.

\subsection{Construction of periodic orbits for the square billiard}
\label{sec:classify}
If the trajectory generated by a pair $\langle P_0, \alpha_0 \rangle$ is periodic, then the unfolded trajectory will intersect shifted copies of the side {\sf AB} infinitely many times at distances $P_0$. The (rational) slope $\tan{(\alpha_0)}$ of this trajectory can be determined using any such shifted copy: if a periodic orbit is obtained after $p$ horizontal and $q$ vertical reflections then the same periodic orbit is obtained after $d \, p$ horizontal and $d \, q$ vertical reflections, for any $d \in \mathbb{N}$, because $\tan{(\alpha_0)} = \frac{p}{q} = \frac{d \, p}{d \, q}$. For the same reason, if $d \in \{ 3, \ldots,\min{(p, q)} \}$ divides both $p$ and $q$ then there exists an intersection at distance $P_0$ on a shifted copy of the side {\sf AB} that lies closer to the original {\sf AB}. Of course, $d = 2$ divides $p$ and $q$, because they are both even, but the corresponding point $(P_0 + \frac{p}{2},\, \frac{q}{2})$ does not necessarily lie on the side {\sf AB}, or if it does, this side may have the reflected orientation {\sf BA} instead. Inspired by Baxter and Umble~\cite{baxter} who give a similar definition for the equilateral triangle, we formalise the notion of \emph{least} period of a periodic unfolded trajectory.
\begin{definition}
\label{def:square_least}
Given an unfolded trajectory on $\mathbb{R}^2$ that is periodic after $p$ horizontal and $q$ vertical reflections. Then its \emph{least} period is $p+q$ if $\gcd{(p, q)} = 2$.
\end{definition}
To generate a periodic orbit with least period $K \in 2 \, \mathbb{N}$, we first construct the corresponding unfolded trajectory in $\mathbb{R}^2$. This unfolded trajectory is given by a line with slope $\frac{p}{q}$, where $p \in \mathbb{N}$ and $q \in \mathbb{N}$ are chosen such that $\gcd{(p, q)} = 2$ and $p + q = K$. In fact, since both $p$ and $q$ are even, we can look for $m = \frac{p}{2}$ and $n = \frac{q}{2}$ such that $n$ and $m$ are co-prime and $m + n = \frac{K}{2}$; indeed, $\frac{p}{q} = \frac{m}{n}$ and $\gcd{(p, q)} = 2 \, \gcd{(m, n)} = 2$. 

For example, to generate a periodic orbit with least period $K = 4$, e.g., as shown in \fref{fig:square_examples}(b), we must find $m, n \in \mathbb{N}$ with $\gcd{(n, m)} = 1$ such that $m + n = \frac{K}{2} = 2$. The only such candidate is $n = m = 1$. Therefore, any period-four orbit is generated by a pair $\langle P_0, \alpha_0 \rangle$ with $\alpha_0 = \atan{(1)} = \frac{\pi}{4}$. Here, the choice for $P_0 \in [0, 1)$ is arbitrary, except that we must avoid generating a singular orbit; for the period-four orbit, it is perhaps not hard to see that the only restriction is $P_0 \neq 0$, but see already \sref{sec:gendiag} for details. Note that, for any other even period $K$, two easy choices are $m = 1$ and $n = \frac{K}{2} - 1$, or vice versa, which both satisfy Definition~\ref{def:square_least}. Hence, as long as $m \neq n$, this trivial decomposition already generates two different periodic orbits with the same period: one periodic orbit is of type $(2,\, 2 n)$ and the other of type $(2 m,\, 2)$, which means, according to Definition~\ref{def:POclass}, that these periodic orbits are from the two different equivalence classes $\mathcal{C}_K(2)$ and $\mathcal{C}_K(2 m) = \mathcal{C}_K(K - 2)$, respectively.

For sufficiently large $K$, we expect there to be other pairs of co-prime numbers that sum to $\frac{K}{2}$, that is, we expect there to be more equivalence classes and different types of periodic orbits with the same period $K \in 2 \, \mathbb{N}$. To determine how many possible decompositions there are, we use Euler's totient function from number theory~\cite{totient}. For any $N \in \mathbb{N}$, Euler's totient function $\varphi(N)$ counts the number of natural numbers $n \in \{ 1, \ldots, N \}$ such that $\gcd{(n, N)} = 1$. The integers $n$ that satisfy this property are referred to as \emph{totatives} of $N$. 
\begin{proposition}
\label{prop:eulertotient}
The total number of ordered pairs $(m, n)$ with $m, n \in \mathbb{N}$ such that $\gcd{(m,n)} = 1$ is equal to Euler's totient function $\varphi(N)$ evaluated at $N = m+n$.
\end{proposition}
%
\begin{proof}
Consider $m, n \in \mathbb{N}$ with $\gcd{(m, n)} = 1$ and define $N = m + n$. By definition, $\gcd{(m, n)} = \gcd{(m,\, m + n)} = \gcd{(m, N)}$. However, $\gcd{(m, n)} = 1$, so $\gcd{(m, N)} = 1$, which means that $m$ is a totative of $N$. Therefore, each ordered pair $(m, n)$ is determined by whether $m$ is a totative of $N$, and the total number of possibilities is $\varphi(N)$.
\end{proof}
The total of $\varphi(N)$ different pairs $(m, n)$ lead to $\varphi(N)$ different ratios $\frac{m}{n}$, so that we may conclude the following.
\begin{corollary}
\label{cor:square_unique_angles}
For $K = 2 \, N$, with $N \in \mathbb{N}$, there are $\varphi(N)$ different types of period-$K$ orbits for the square billiard, generated by $\varphi(N)$ different angles in the interval $(0, \frac{\pi}{2}]$, which represent a total of $\varphi(N)$ different equivalence classes.
\end{corollary}
%
\begin{figure}[t!]
  \centering
  \includegraphics{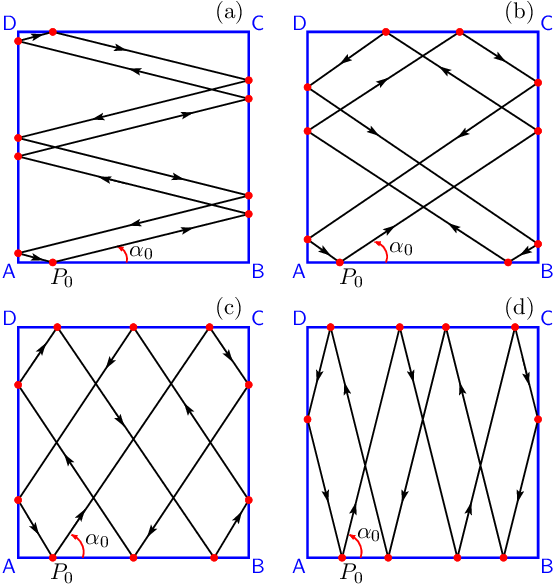}
  \caption{\label{fig:period10}
    Period-ten orbits generated from $P_0 = 0.15$ for the square $\titlebilliard$ from each of the four equivalence classes. The periodic orbits in panels~(a)--(d) are of type $(2, 8)$, $(4, 6)$, $(6, 4)$, and $(8, 2)$, respectively.}
\end{figure}
%
For example, there exist $\varphi(5) = 4$ different types of period-ten orbits. Note that $5$ is prime, so the totatives of $5$ are all integers $1$, $2$, $3$, $4$ less than $5$. Therefore, there are four possible ordered pairs, namely, $(1, 4)$, $(2, 3)$, $(3, 2)$, and $(4, 1)$, leading to four different families of period-ten orbits from the corresponding equivalence classes $\mathcal{C}_{10}(2)$, $\mathcal{C}_{10}(4)$, $\mathcal{C}_{10}(6)$, and $\mathcal{C}_{10}(8)$, respectively. \Fref{fig:period10} illustrates these four different types of period-ten orbits generated by the pairs $\langle P_0, \alpha_0 \rangle$ with $P_0 = 0.15$ fixed and $\tan{(\alpha_0)} = \frac{1}{4}$, $\frac{2}{3}$, $\frac{3}{2}$, and $\frac{4}{1}$, in panels~(a)--(d), respectively. As another example, Corollary~\ref{cor:square_unique_angles} implies that there cannot exist a period-eight orbit of type $(4, 4)$, that is, a periodic orbit that unfolds into a trajectory which repeats after four vertical and four horizontal reflections. Indeed, the corresponding slope would have to be $1$ and, thus, $\alpha_0 = \atan(1) = \frac{\pi}{4}$, which generates a period-four orbit. Hence, such a period-eight orbit would just be a double copy of the period-four orbit; while its period may be eight, the \emph{least} period of this periodic orbit is four.
\begin{remark}
There is no efficient algorithm to find the totatives of a given integer $N \in \mathbb{N}$, because it is equivalent to prime factorisation. Indeed, Euler's totient function is explicitly given by the classical Euler's product formula,
\begin{displaymath}
  \varphi(N) = N \, \prod\limits_{p | N} \left( 1 - \frac{1}{p} \right),
\end{displaymath}
that computes a product over all distinct prime factors of $N$. An alternative formula,
\begin{displaymath}
  \varphi(N) = \sum_{k=1}^{N} \gcd{(k, N)} \cos{ \left( 2 \pi \tfrac{k}{N} \right) },
\end{displaymath}
can be derived using the discrete Fourier transform~\cite{schramm_gcd}; unfortunately, finding the greatest common divisors $\gcd{(k,N)}$ is also computationally as complex as finding prime factors.
\end{remark}
%

\subsection{Singular orbits for the square billiard}
\label{sec:gendiag}
So far, we have not imposed any conditions on the choice for $P_0 \in [0, 1)$. Indeed, with our definition of equivalence class, a periodic orbit generated by $\langle P_0, \alpha_0 \rangle$ is equivalent to other periodic orbits generated by the same initial angle $\alpha_0$, but starting from a different initial point $P_0$. However, we cannot choose just any value for $P_0$. For example, we have already seen that $P_0 = 0$ leads to a singular rather than a periodic orbit. Perhaps there are other values $P_0 \in [0, 1)$ for which the pair $\langle P_0, \alpha_0 \rangle$ generates a billiard trajectory that terminates at one of the vertices? To answer this question, let us first consider singular orbits in general.
\begin{proposition}
\label{prop:square_gendiag}
If a billiard trajectory starts at a vertex with an irrational slope then it never terminates at a vertex on the square billiard.
\end{proposition}
\begin{proof}
Suppose for the sake of contradiction that there exists a singular orbit that starts from a vertex with irrational slope. Without loss of generality, we may assume that this billiard trajectory starts at vertex {\sf A} of the square $\billiard$; if not, we rotate the table and relabel the vertices. We now consider the corresponding unfolded trajectory in the plane $\mathbb{R}^{2}$. Since we assumed that the billiard trajectory is singular, the unfolded trajectory must pass through a shifted copy of one of the vertices; by Definition~\ref{def:unfolding_orientation}, these shifted copies lie in $\mathbb{Z} \times \mathbb{Z}$. Therefore, the unfolded trajectory contains a line segment from the original vertex {\sf A} at $(0, 0)$ to a shifted copy of this same or another vertex at the point, say, $(m, n) \in \mathbb{Z} \times \mathbb{Z}$; here both $n \neq 0$ and $m \neq 0$, because the slope is neither $0$ nor $\infty$. This means that the unfolded trajectory is the straight line with slope $\frac{m}{n} \in \mathbb{Q}$. Any possible rotation back to the original billiard table either preserves this slope (rotation by $0$ or $\pi$), or changes it to $\frac{n}{m} \in \mathbb{Q}$ (rotation by $\pm \frac{\pi}{2}$), both of which are rational; a contradiction.
\end{proof}
\begin{corollary}
\label{cor:square_gendiag}
If a billiard trajectory for the square billiard $\billiard$ that starts at a point $P_0 \neq {\sf A}$ with a rational slope is a singular orbit then the trajectory generated from $P_0$ in the opposite direction, using the reflected (negative) slope, will also be singular. 
\end{corollary}
Indeed, both billiard trajectories unfold to the same line with rational slope in the plane $\mathbb{R}^2$. Hence, every singular orbit is contained in a billiard trajectory that both starts and terminates at a vertex. Such billiard trajectories are called \emph{generalised diagonals}~\cite{katok1987}.
\begin{definition}
\label{def:square_gendiag}
A generalised diagonal is a singular orbit that starts at a vertex. The length of a generalised diagonal is given by the total number of (non-vertex) collision points.
\end{definition}
%

\begin{figure}[t!]
  \centering
  \includegraphics{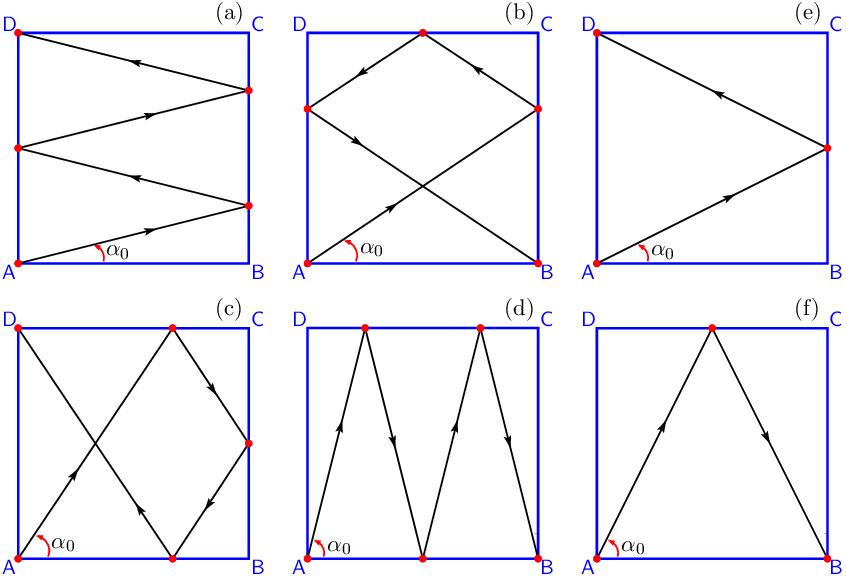}
  \caption{\label{fig:gendiag}
    Examples of generalised diagonals. The panels in the first two columns, labelled (a)--(d), show the four different types of generalised diagonals of length three, and the right-most two panels labelled (e) and (f) show the two types of length two; compare also with \fref{fig:period10}.}
\end{figure}
%
We now return to the question which values $P_0 \in [0, 1)$ for the pair $\langle P_0, \alpha_0 \rangle$ with $\tan{(\alpha_0)}$ rational generate periodic rather than singular orbits. As an example, consider the four different period-ten orbits shown in \fref{fig:period10}. Notice that the periodic orbit in panel~(a), which starts at $P_0 = 0.15$ with slope $\frac{1}{4}$, has five pairs of collision points that are located close together: two pairs on the side {\sf BC}, one pair on {\sf DA}, one pair near vertext {\sf A} on sides {\sf AB} and {\sf DA}, and a fifth pair near vertex {\sf D} on sides {\sf DA} and {\sf CD}. If we shift $P_0$ towards vertex {\sf A}, these pairs of collision points will move even closer together, until each pair merges as $P_0$ reaches {\sf A}. The resulting billiard trajectory, which is shown in \fref{fig:gendiag}(a), is a generalised diagonal that starts at vertex {\sf A} and terminates at vertex {\sf D} after three collisions; hence, it has length three. A similar shift of $P_0$ to vertex {\sf A} for the other three types of period-ten orbits, shown in panels~(b)--(d) of \fref{fig:period10}, leads to different generalised diagonals that start at {\sf A} with slopes $\frac{2}{3}$, $\frac{3}{2}$, and $4$, respectively; these are shown in panels~(b)--(d) of \fref{fig:gendiag}, respectively. (We chose a somewhat pecular labelling of the panels in \fref{fig:gendiag} such that the labels of the first two columns match the labels of the panels in \fref{fig:period10}.) Note that each of these generalised diagonals also has length three. Analogous to Definition~\ref{def:POclass} of equivalent periodic orbits, the four generalised diagonals in \fref{fig:gendiag}(a)--(d) are from different equivalence classes, because they have, respectively, $0$, $1$, $2$ and $3$ of their three collision points on the horizontal sides of the square billiard.

\Frefs{fig:gendiag}(e) and~\ref{fig:gendiag}(f) show the generalised diagonals obtained when we use the same approach for the period-six orbits of types $(2, 4)$ and $(4, 2)$ in \fref{fig:equivalence}(c) and~\ref{fig:equivalence}(a), respectively. These generalised diagonals each have length one and have $0$ or $1$ collision point on the horizontal sides of the square billiard.
\begin{proposition}
\label{prop:gendiag_length}
Consider a generalised diagonal for the square billiard $\billiard$ that starts at vertex {\sf A} with rational slope $\frac{m}{n}$, where $\gcd{(m, n)} = 1$. Then this generalised diagonal has length $m + n - 2$ given by $m - 1$ collision points on the horizontal sides {\sf AB} or {\sf CD} and $n - 1$ collision points on the vertical sides {\sf BC} or {\sf DA}. 
\end{proposition}
\begin{proof}
The unfolded trajectory of such a generalised diagonal is the line through the origin with slope $\frac{m}{n}$ that passes through all vertex points $(i, j) \in \mathbb{Z} \times \mathbb{Z}$ that are integer multiples of $(n, m)$. There are no other points $(i, j) \in \mathbb{Z} \times \mathbb{Z}$ on this line, because $\gcd{(m, n)} = 1$. From $(0,0)$ to $(n, m)$, we encounter exactly $m-1$ vertical reflections and exactly $n-1$ horizontal reflections, which means that there are $m - 1$ collision points on the horizontal sides {\sf AB} or {\sf CD} and $n - 1$ collision points on the vertical sides {\sf BC} or {\sf DA}. The length of this singular trajectory is, thus, $m - 1 +  n - 1 = m + n - 2$, as required. 
\end{proof}
Observe that a generalised diagonal can never start and terminate at the same vertex. Ineed, if a generalised diagonal for the square $\billiard$ starts, for example, at vertex {\sf A} then it can only terminate at {\sf A} if the unfolded trajectory is a line that passes through one of the points $(i, j) \in \mathbb{Z} \times \mathbb{Z}$ with both $i$ and $j$ even; we may assume $\gcd{(i, j)} = 2$, because otherwise $(i, j)$ is not the first vertex at which the generalised diagonal terminates. The unfolded trajectory is then a line with slope $\frac{j}{i} = \frac{j/2}{i/2}$, which passes through the vertex point $(\frac{i}{2}, \frac{j}{2})$ before reaching {\sf A}. Since $\gcd{(i, j)} = 2$, either $\frac{i}{2}$ or $\frac{j}{2}$ or both will be odd, which means that the generalised diagonal already terminated in the vertex {\sf B}, {\sf D}, or {\sf C}, respectively. 
\begin{remark}
Proposition~\ref{prop:gendiag_length}, combined with Proposition~\ref{prop:eulertotient}, enables us to list all generalised diagonals of a given length that start at vertex {\sf A}. In total, there are twice the number of different generalised diagonals, because we consider a billiard trajectory equal to its reversed-direction copy. More precisely, there are exactly two generalised diagonals in each equivalence class, namely, a generalised diagonal that starts at vertex {\sf A} and terminates at one of the other vertices, and a reflected copy of this generalised diagonal that starts and terminates at the other (remaining) two vertices. 
\end{remark}

\Frefs{fig:equivalence}, \ref{fig:period10}, and \ref{fig:gendiag} suggest a relation between the equivalence class of a periodic orbit in the square billiard of a particular period and the types and lengths of the generalised diagonals that bound the families in this class. Indeed, we find that all families of periodic orbits of a given type are related to the same generalised diagonal. More precisely, we have the following result.
\begin{theorem}
\label{thm:gendiag}
Given $p, q \in \mathbb{N}$ such that $\gcd{(p, q)} = 2$, consider the billiard trajectory for the square billiard generated by a pair $\langle P_0, \alpha_0 \rangle$ with $\tan{(\alpha_0)} = \frac{p}{q}$ fixed and $P_0 \in [0, 1)$ varying. Then the billiard trajectory is one of the following:
\begin{displaymath}
  \left\{ \begin{array}{ll}
            \mbox{a singular orbit}, & \mbox{if } P_0 = \frac{2 \ell}{p} \mbox{ for } \ell = 0, 1, \ldots, \frac{p}{2} - 1, \\
            \mbox{a periodic orbit in } \mathcal{C}_{p+q}(p), & \mbox{otherwise},
          \end{array} \right.
\end{displaymath}
and the singular orbit lies on a generalised diagonal of length $\frac{p+q}{2}-2$ with precisely $\frac{p}{2} - 1$ collision points on the horizontal sides {\sf AB} or {\sf CD} and $\frac{q}{2} - 1$ collision points on the vertical sides {\sf BC} or {\sf DA}. 
\end{theorem}
\begin{proof}
Since $\tan{(\alpha_0)} = \frac{p}{q}$ is rational, this billiard trajectory cannot be a non-periodic orbit, because non-periodic orbits unfold to lines with irrational slopes; cf.\ Remark~\ref{rem:dichotomy}. If it is periodic, then it is of type $(p, q)$ and Definitions~\ref{def:POclass} and~\ref{def:square_least} imply that such a periodic orbit will be a member of the class $\mathcal{C}_K(p)$, with $K = p + q$. Furthermore, the billiard trajectory unfolds to a line (or line segment) with fixed slope $\frac{p}{q}$, so if it is singular, it must lie on a generalised diagonal with this slope. Since $\gcd{(p, q)} = 2$, Proposition~\ref{prop:gendiag_length} implies that the generalised diagonal will have length $\frac{p}{2} + \frac{q}{2} - 2$ as required, and it will, indeed, have $\frac{p}{2} - 1$ collision points on the horizontal sides {\sf AB} or {\sf CD} and the remaining $\frac{q}{2} - 1$ collision points on the vertical sides {\sf BC} or {\sf DA}. Hence, the proof is complete as soon as we show that a billiard trajectory generated by the pair $\langle P_0, \alpha_0 \rangle$ is singular if and only if $P_0 = \frac{2 \ell}{p}$ for $\ell = 0, 1, \ldots, \frac{p}{2} - 1$. 

Consider the unfolding of a generalised diagonal that starts at vertex {\sf A} with slope $\frac{m}{n} = \frac{p}{q}$. This generalised diagonal terminates at the point $(n, m) = (\frac{q}{2}, \frac{p}{2}) \in \mathbb{Z} \times \mathbb{Z}$, because $\gcd{(p, q)} = 2$, so $\frac{m}{n} = \frac{p}{q}$ and $\gcd{(m, n)} = 1$. The $\frac{p}{2} - 1$ collision points on the horizontal sides {\sf AB} or {\sf CD} are given by the $m - 1$ intersection points with the horizontal lines $\{ y = j \}$ for $j = 1, \ldots, m - 1$. If we translate this line down over $j$ integer units, and left over another, say, $i$ integer units for some $0 \leq i < n$, we can move any one of these intersection points to the segment $[0, 1]$ on the $x$-axis, which corresponds to the original side {\sf AB} of the square billard. The translated points still lie on a generalised diagonal, because both the start and terminal vertex points at $(0, 0)$ and $(n, m)$ map to $(-i, -j)$ and $(n - i,\, m - j)$, respectively, which are also vertex points. The corresponding values for $P_0$ are then given by the distances to the origin of the translated intersection points, or equivalently, the distances to the vertices $(i, j)$ for the intersection points on the lines $\{ y = j \}$, with $j = 1, \ldots, m - 1$. Note that there can be no other values for $P_0 \in [0, 1)$ that lead to singular orbits, because their unfolded trajectories must lie on generalised diagonals with slope $\frac{m}{n} = \frac{p}{q}$ that have at most $m - 1$ collision points on the side {\sf AB}; Hence, if we include $P_0 = 0$, there are only $m$ candidates. The intersection points with integer $y$-coordinates $y = j$ have $x$-coordinates $x = \frac{n}{m} \, j$, so the values for $P_0$ that lead to singular orbits are
\begin{displaymath}
  P_0 = \tfrac{n}{m} j \; \mbox{(mod 1)} \in [0, 1), \mbox{ for } j = 0, \ldots, m - 1.
\end{displaymath}
We claim that this set of points is the same as the set of points $P_0 = \frac{2 \ell}{p}$ for $\ell = 0, 1, \ldots, \frac{p}{2} - 1 = m - 1$. To see this, first note that the points $\frac{n}{m} (j - 1)$ and $\frac{n}{m} j$ differ by $\frac{n}{m}$ for all $j = 1, \ldots, m - 1$; we say that they are uniformly distributed on the interval $[0, n]$. It is important to realise that the first point is located at $0$. Imagine wrapping this interval $n$ times around a circle with circumference $1$. Then each point $\frac{n}{m} j$ will map to a point on this circle at an arclength from $0$ that cannot exceed $1$; each point $\frac{n}{m} j$ maps to a different point on the circle, because $\gcd{(m, n)} = 1$. We can view such points as angles $2 \pi \theta_j$, measured in radians, with $\theta_j \in [0, 1)$. The difference $\frac{n}{m}$ between two neighbouring points $\frac{n}{m} (j - 1)$ and $\frac{n}{m} j$ translates on the circle to the arclength distance $\frac{n}{m} \; \mbox{(mod 1)}$, or a rotation by angle $2 \pi \frac{n}{m} \; (\mbox{mod $2 \pi$)}$, but the points may not be direct neighbours any longer; in other words, the sequence of values $\theta_j$ for $j = 0, \ldots, m-1$ is not necessarily in increasing order. Next, observe that $m \, \tfrac{n}{m} j \in \mathbb{Z}$, so $m \, \tfrac{n}{m} j \; \mbox{(mod 1)} = 0$ for all $j = 0, \ldots m - 1$. In the complex plane, this means that each point $\frac{n}{m} j \; \mbox{(mod 1)}$ is a solution to the equation
\begin{displaymath}
  \left( e^{2 \pi i \theta} \right)^m = 1,
\end{displaymath}
which are determined by the $m$ roots of unity in the complex plane, and these are uniformly distributed on the unit circle. Since $e^0 = 1$, the roots are given by the angles $2 \pi \theta_\ell$, with $\theta_\ell = \frac{\ell}{m} = \frac{1}{m} \ell \in [0, 1)$, for $\ell = 0, \ldots, m - 1$ also uniformly distributed on the unit interval $[0, 1)$. Therefore, the (unordered) set $\left\{ \theta_j \frac{n}{m} j \; \mbox{(mod 1)} \bigm| j = 0, \ldots, m-1 \right\}$ is the same as the (ordered) set $\left\{ \theta_\ell = \frac{\ell}{m} \bigm| \ell = 0, \ldots, m-1 \right\}$, because both sets contain exactly $m$ uniformly distributed points in $[0, 1)$, starting from $\theta_0 = 0$.

For each $j = 1, \ldots, m - 1$, we obtain $\theta_j = \frac{n}{m} j \; \mbox{(mod 1)}$ from the ordered set $\left\{ \theta_\ell = \frac{\ell}{m} \bigm| \ell = 0, \right.$ $\left.1, \ldots, m-1 \right\}$ by taking the $n$th neighbour after the point $\theta_\ell$ that corresponds to $\theta_{j-1}$; here, we treat $\theta_0 = 0$ as a neighbour of $\theta_{m-1} = \frac{m - 1}{m}$. \end{proof}
\begin{remark}
The expert reader will recognise the sequence $\theta_j$ as a trajectory of the circle map on the unit interval defined by the rigid rotation $x \mapsto x + \omega$ with $\omega = \frac{m}{n}$ and winding number $n$. 
\end{remark}
%
\begin{figure}[t!]
  \centering
  \includegraphics{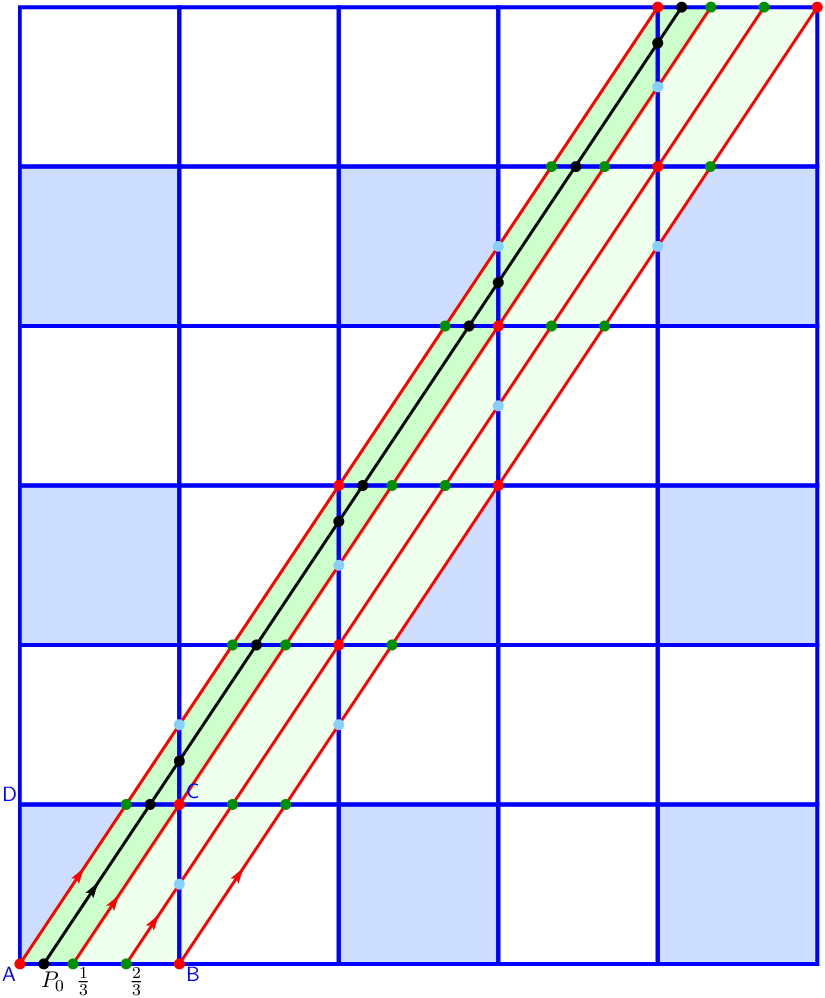}
  \caption{\label{fig:POfamilies} 
    Unfolding of all possible period-ten orbits of type $(6, 4)$, and associated bounding singular orbits that start on the side {\sf AB} of the square $\titlebilliard$. Shown are  the period-ten orbit (black) from \fref{fig:period10}(c) that starts at $P_0 = 0.15$ and the generalised diagonal (red) from \fref{fig:gendiag}(c) starting at vertex {\sf A}, together with three translated versions that give the two singular orbits starting at $\frac{1}{3}$, $\frac{2}{3}$ and the generalised diagonal starting at vertex {\sf B}. The dark (green) shaded strip represents the entire family in $\mathcal{C}_{10}(6)$, while the light (green) shaded strip indicates the regime of existence starting from any point on the side {\sf AB}; the (light-blue) shaded tiles correspond to tables with the orientation $\titlebilliard$.} 
\end{figure}
%
To illustrate Theorem~\ref{thm:gendiag}, consider the period-ten orbit from \fref{fig:period10}(c), which is a member of the class $\mathcal{C}_{10}(6)$ and generated by the pair $\langle 0.15,\, \atan{(\frac{3}{2})} \rangle$. Its corresponding generalised diagonal of length three is shown in \fref{fig:gendiag}(c). \Fref{fig:POfamilies} shows the unfoldings of these two billiard trajectories as two lines with slope $\frac{3}{2}$: the generalised diagonal is the (red) line starting at vertex {\sf A} at  the origin (red point), and the periodic orbit is the (black) line starting at the point $(0.15, 0)$ (black), respectively. Here, we used a (light-blue) shading for the tiles in $\mathbb{R}^2$ with orientation $\billiard$, instead of labelling each reflected vertex. The generalised diagonal has three collision points, which are consecutively located at the points $(\frac{2}{3}, 1)$ (green), $(1, \frac{3}{2})$ (light-blue), and $(\frac{4}{3}, 2)$ (green) in the plane. The first and last (green) points are collisions with the sides {\sf AB} or {\sf CD} and they determine the possible initial points $P_0$ on {\sf AB} that lead to a period-ten orbit when the initial slope of the billiard trajectory is $\frac{3}{2}$. The initial points that are excluded lie on the singular orbits (red lines) obtained by translation of the generalised diagonal through {\sf A} such that each of the two collision points $(\frac{2}{3}, 1)$, $(\frac{4}{3}, 2)$ are mapped to the interval $[0, 1)$ on the $x$-axis. For $(\frac{2}{3}, 1)$, this is achieved by starting the generalised diagonal from the vertex $(0, -1)$, which crosses the $x$-axis at $(\frac{2}{3}, 0)$, and for $(\frac{4}{3}, 2)$ the generalised diagonal should start from vertex point $(-1, -2)$ so that it crosses the $x$-axis at the point $(\frac{1}{3}, 0)$. Hence, all period-ten orbits in the class $\mathcal{C}_{10}(6)$ are generated by pairs $\langle P_0, \atan{(\frac{3}{2})} \rangle$ with $P_0 \in (0, \frac{1}{3})$, $P_0 \in (\frac{1}{3}, \frac{2}{3})$, or $P_0 \in (\frac{2}{3}, 1)$, while starting points $P_0 = \frac{1}{3}$ and $P_0 = \frac{2}{3}$ generate singular orbits with that angle. There are no other period-ten orbits that collide six times with the horizontal and four times with the vertical sides of the square $\billiard$.

Note that each period-ten orbit in the class $\mathcal{C}_{10}(6)$ has three collision points on the side {\sf AB}. Each such period-ten orbit is uniquely identified by the generator pair $\langle P_0,\, \atan{(\frac{3}{2})} \rangle$ with $P_0 \in (0, \frac{1}{3})$. Indeed, any period-ten orbit generated by the pair $\langle P,\, \atan{(\frac{3}{2})} \rangle$ with $P \in (\frac{2}{3}, 1)$ is exactly the same period-ten orbit as the one starting from $P_0 = P  - \frac{2}{3} \in (0, \frac{1}{3})$, which encounters $P \in (\frac{2}{3}, 1)$ as its sixth collision point; see \fref{fig:POfamilies}. Similarly, any period-ten orbit that starts from $P \in (\frac{1}{3}, \frac{2}{3})$ is the same as the one starting from $P_0 = \frac{2}{3} - P$, which encounters the point $P \in (\frac{1}{3}, \frac{2}{3})$ on {\sf AB} after three collisions; this is shown in \fref{fig:POfamilies} on an (unshaded) tile with the vertically reflected orientation $\billiard[BCDA]$. Hence, the family of period-ten orbits in the class $\mathcal{C}_{10}(6)$ is completely represented by an initial point $P_0$ up to distance $\frac{1}{3}$ from {\sf A} along the side {\sf AB} and the outgoing line with slope $\frac{3}{2}$.

\section{Rectangular billiards}
\label{sec:rectangle}
So far, we restricted attention to the square billiard, that is, a rectangular table with aspect ratio $1:1$. In this section, we extend our results to rectangular tables with arbitrary aspect ratios. More precisely, we consider a rectangular table represented by the scaled table $\billiard$ with sides {\sf AB} and {\sf CD} of length $1$ and sides {\sf BC} and {\sf DA} of length $\varrho$, where $0 < \varrho < \infty$. Recall that, just as for the square billiard, a constant scaling of the sides of the table does not change the number or types of periodic orbits. The question addressed in this section is whether a scaling of just the sides {\sf BC} and {\sf DA} changes the number or types of periodic orbits.
\begin{remark}
For the square table, it is straightforward to rotate the table and place any of the four vertices at the origin. For the rectangular table with aspect ratio $1:\varrho$, this hold for vertices {\sf A} and {\sf C} only; the quarter rotations that place either vertex {\sf B} or vertex {\sf D} at the origin result in a scaled table with aspect ratio $1:\frac{1}{\varrho}$. This is the reason why we made a distinction between periodic orbits of type $(p, q)$ and those of type $(q, p)$, or equivalently, between the equivalence classes $\mathcal{C}_K(p)$ and $\mathcal{C}_K(q) = \mathcal{C}_K(K - p)$. 
\end{remark}
The rectangular table is a natural and perhaps most straightforward extension from the square table, because it is the generic linear transformation of the square table that preserves the angles between all four sides. Consequently, the technique of unfolding a billiard trajectory to a straight line in the plane $\mathbb{R}^2$ can be applied in the same way as for the square billiard; essentially, the unfolding takes place relative to the vertices that now lie on the transformed mesh $\mathbb{Z} \times \varrho \, \mathbb{Z}$, rather than the original square mesh $\mathbb{Z} \times \mathbb{Z}$. This also means that there exists a one-to-one correspondence between billiard trajectories on the square billiard and those on the rectangular billiard. More precisely, any billiard trajectory on the square, be it singular, periodic, or non-periodic, can be unfolded and then mapped via the appropriate linear transformation to a straight line on the transformed mesh $\mathbb{Z} \times \varrho \, \mathbb{Z}$ that again corresponds to a singular, periodic, or non-periodic orbit on the rectangular billiard, respectively. Here, the slope of the unfolded trajectory will change (by a factor $\varrho$), but the line is positioned the same way relative to the vertices. 

\begin{figure}[t!]
  \centering
  \includegraphics{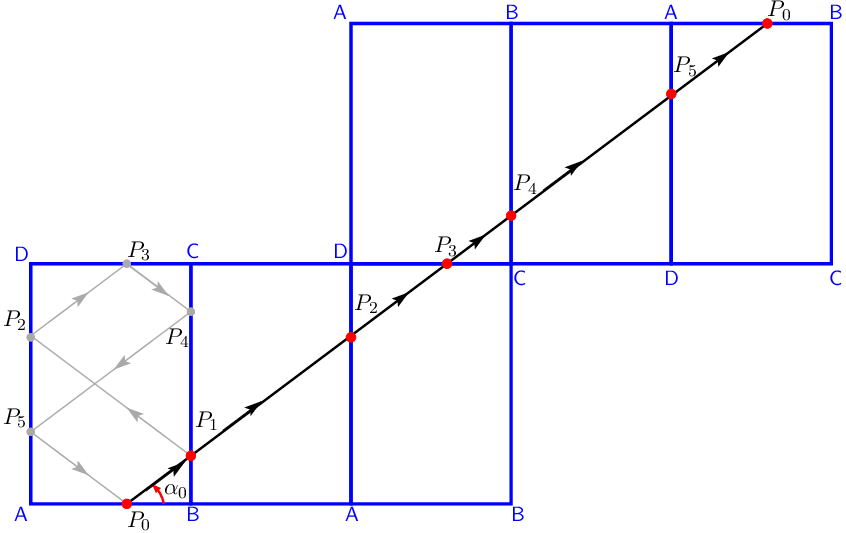}
  \caption{\label{fig:rec_periodicity} 
    Unfolding of a period-six orbit of type $(2,4)$ in the rectangular billiard $\titlebilliard$ with aspect ratio $1:\frac{3}{2}$; compare with \frefs{fig:equivalence}(c) and~\ref{fig:square_unfolding}.} 
\end{figure}
%
As an example, compare the period-six orbit from \frefs{fig:equivalence}(c) and~\ref{fig:square_unfolding} with the period-six orbit shown in \fref{fig:rec_periodicity} that lies in the rectangle $\billiard$ with aspect ratio $1:\varrho$, where $\varrho = \frac{3}{2}$. As in \fref{fig:square_unfolding}, the trajectory has been unfolded on the plane as well. Observe that the sides {\sf BC} and {\sf DA} of the rectangular billiard are $1.5$ times longer than its sides {\sf AB} and {\sf CD}. Hence, the line that connects the point $P_0 = (0.7, 0)$ with its translated copy after a total of two vertical and four horizontal reflections has slope $\frac{3}{4} = \varrho \frac{2}{4} = \varrho \frac{1}{2}$; hence, the slope of the unfolded trajectory on the rectangular table changes by a factor $\varrho$ compared to the slope $\frac{1}{2}$ for the line in \fref{fig:square_unfolding} associated with the square table. 

Other than the adjustment by this factor $\varrho$, the results for the square billiard presented in the previous section carry over in a natural way to the rectangular billiard with aspect ratio $1:\varrho$. We briefly summarise these extended results here; note that the characterisation includes the square billiard as the case $\varrho = 1$.

\subsection{Classification of periodic orbits for the rectangular billiard}
The rectangular billiard has the same equivalence classes of periodic orbits, which can be constructed in the same way after adjusting the slope as required for the given aspect ratio. In particular, all periodic orbits for the rectangular billiard have even period, and the product of the slope of the trajectory and aspect ratio of the rectangle is rational; see also Theorem~\ref{thm:square_periodic}. More precisely, we have the following extended version of Theorem~\ref{thm:gendiag}.
\begin{proposition}
\label{prop:rectangle_periodic}
For the rectangular billiard $\billiard$ with aspect ratio $1:\varrho$, there exists a periodic orbit that has $p$ distinct collision points on the sides {\sf AB} or {\sf CD} and $q$ distinct collision points on the sides {\sf BC} or {\sf DA} for any $p, q \in \mathbb{N}$ with $\gcd{(p, q)} = 2$. This periodic orbit is from the equivalence class ${\cal C}_K(p)$ with $K = p + q$, and it is generated by a pair $\langle P_0, \alpha_0 \rangle$, with $\alpha_0 \in (0, \frac{\pi}{2}]$ such that $\tan{(\alpha_0)} = \varrho \, \frac{p}{q}$. The point $P_0$ can take almost any value in the interval $[0, 1)$; the exceptions given by 
\begin{displaymath}
            P_0 = \frac{2 \ell}{p} \mbox{ for } \ell = 0, 1, \ldots, \frac{p}{2} - 1, \\
\end{displaymath}
lead to a singular orbit, which lies on a generalised diagonal of length $\frac{p+q}{2}-2$ with precisely $\frac{p}{2} - 1$ collision points on the horizontal sides {\sf AB} or {\sf CD} and $\frac{q}{2} - 1$ collision points on the vertical sides {\sf BC} or {\sf DA}. 
\end{proposition}
Given $K = 2 N$ for some $N \in \mathbb{N}$, we can construct period-$K$ orbits of type $(p, q)$ for any $p, q \in \mathbb{N}$ such that $K = p + q$ and $\gcd{(p, q)} = 2$. Hence, just as for the square billiard we can apply Proposition~\ref{prop:eulertotient} which means that Corollary~\ref{cor:square_unique_angles} also holds for the rectangular billiard: there are $\varphi(N)$ unique combinations for $p$ and $q$, leading to $\varphi(N)$ different types of period-$K$ orbits, generated by $\varphi(N)$ different angles in the interval $(0, \frac{\pi}{2}]$, which represent a total of $\varphi(N)$ different equivalence classes.

It is important to realise that any rectangle will have periodic orbits of all (even) periods, including the periodic orbits with period $K = 2$, which are not covered in Proposition~\ref{prop:rectangle_periodic}. The period-two orbits are kind of limiting cases: if we take $p=2$ and $q=0$ then $\tan{(\alpha_0)} = \varrho \, \frac{p}{q}$ is not defined, but the (vertical) line with infinite slope, that is the angle $\alpha_0 = \frac{\pi}{2}$ produces a period-two orbit from ${\cal C}_2(2)$ for all rectangles, and for any initial point $P_0 \in (0, 1)$; similarly, with $p=0$ and $q=2$, we get $\alpha_0 = 0$, which is not allowed, but does lead to period-two orbits from ${\cal C}_2(0)$ that only collide with the vertical sides {\sf BC} and {\sf DA}. We generate such a period-two orbit by choosing the initial point $P_0$ anywhere on the side {\sf BC}, except the vertices {\sf B} and {\sf C} themselves.

We also remark that the initial angle $\alpha_0$ is only unique for period-two orbits; its value will vary for other equivalence classes, because of the dependence on the aspect ratio of the billiard table. The converse is also true: for example, the angle $\alpha_0$ with $\tan{(\alpha_0)} = \frac{1}{2}$ generates a period-six orbit for the square billiard table, but this will be a period-ten orbit on the rectangle with aspect ratio $1:2$, while it is a period-four orbit on the rectangle with aspect ratio $1:\frac{1}{2}$, or equivalently, by starting a billiard trajectory with this slope on a vertical side of length $2$. Similarly, this same angle produces periodic orbits of periods ten or 22 on a rectangle with aspect ratio $1:\frac{3}{4}$ when starting from a horizontal or a vertical side, respectively; the period-ten orbit is of type $(4, 6)$, in contrast to the period-ten orbit produced for this angle on a table with aspect ratio $1:2$, which will be of type $(2, 8)$. 

Proposition~\ref{prop:square_gendiag} and Corollary~\ref{cor:square_gendiag} hold for rectangular billiards as well. More precisely, we have the following properties. 
\begin{proposition}
For the rectangular billiard $\billiard$ with aspect ratio $1:\varrho$, consider a billiard trajectory that starts at vertex {\sf A}, {\sf B}, {\sf C}, or {\sf D}. Define $\sigma \in \mathbb{R}$ as the slope of the billiard trajectory measured relative to the side {\sf AB}.
\begin{itemize}
\item if $\frac{\sigma}{\varrho}$ is irrational, then this billiard trajectory never collides with another vertex. \\[1mm]
\item if $\frac{\sigma}{\varrho}$ is rational, then this billiard trajectory will always end up in a vertex, that is, it is a singular orbit and the trajectory generated in the opposite direction, using the reflected (negative) slope, will also be singular. 
\end{itemize}
\end{proposition}
Note that it is necessary to specify the side with respect to which the slope is measured; if the slope is $\sigma$ when measured with respect to the side {\sf AB} then it is also $\sigma$ when measured with respect to the side {\sf CD}, but it will be $\frac{1}{\sigma}$ when measured with respect to the sides {\sf BC} or {\sf DA}. For the square table, this difference does not matter, because $\frac{1}{\sigma}$ is (ir)rational if and only if $\sigma$ is (ir)rational. However, for the rectangular billiard, the adjustment by the factor $\varrho$ does not necessarily preserve this equivalence if $\varrho$ is itself irrational. More precisely, if the slope is $\sigma$ when measured with respect to, say, the side {\sf BC} then we effectively consider the rotated table with aspect ratio $1:\frac{1}{\varrho}$, and the behaviour of the billiard trajectory starting at vertex {\sf B} will be determined by whether $\frac{1/\sigma}{\varrho} = \frac{1}{\varrho \, \sigma}$ is irrational or not; this is not equivalent to asking whether $\frac{\sigma}{\varrho}$ is irrational or not.

\section{Discussion}
\label{sec:conclusion}
We completely classified the existence and nature of all periodic orbits for a rectangular billiard table with aspect ratio $1:\varrho$, where $0 < \varrho < \infty$; we discussed the square billiard table in full detail, which is included in this setting as the case with $\varrho = 1$. The class of rectangular billiars is special, because it is the only class of polygonal billards for which the classification of all periodic orbits is preserved under a non-uniform scaling~\cite{davispreprint, duchinpreprint}. The special property that the sides of a rectangular table are perpendicular to each other enables us to characterise the periodic orbits for this class of billard tables in unprecedented detail. In particular, we show that rectangular billiards admit period-$K$ orbits for any even period, but not for any odd period $K \in \mathbb{N}$. As soon as $K \geq 6$, there are different types of period-$K$ orbits, determined by the difference between the numbers of collisions with the two pairs of parallel sides. Each type is completely determined by the slope that the periodic orbit should have with respect to one of the sides of the rectangle. Furthermore, the periodic orbit can be realised with any initial point on this side, except for a finite set of exactly $\frac{1}{2}K - 1$ points, other than the table corners, that give rise to singular orbits; the unfolding of such singular orbits leads to generalised diagonals in the plane. We defined equivalence classes for each type and proved that the total of different types of period-$K$ orbits, for any $K \in 2 \, \mathbb{N}$ is given by Euler's totient function $\varphi(N)$ evaluated at $N = \frac{1}{2}K$; the totients $n \in \{ 1, \ldots, N \}$ such that $\gcd{(n, N)} = 1$ define pairs $(m, n)$ with $m = N - n$ that uniquely define the slopes needed to generate a periodic orbit of a specific type.

As mentioned, the difference between the square and rectangular billiard is a stretching of one pair of parallel sides, which leads to an adjustment by the aspect-ratio parameter $\varrho$ of the slopes required to generate particular periodic orbits. It is tempting to apply a more general transformation, for example, one that scales and shears the sides, such that the billiard table becomes a parallelogram. Unfortunately, the unfolding technique generally fails to tile the plane for parallelogram billiards~\cite{FoxKershner1936, KatokZemlyakov1975, sergebook}. We find that there is very little known regarding parallelogram billiards. Do parallelograms also only admit periodic orbits with even periods? What are necessary and sufficient conditions for existence of a period-$K$ orbit and how do we construct such a periodic orbit for the parallelogram?

The non-existence of periodic orbits with odd periods has been proven for parallelograms consisting of two equilateral triangles glued together, that is, one angle is $\frac{\pi}{3}$ radians~\cite{60_deg_par}. Furthermore, it is known that periodic orbits for the parallelogram with angles $\frac{\pi}{4}$ and $\frac{3\pi}{4}$ do not persist under perturbations, that is, so-called stable periodic orbits do not exist for this parallelogram table~\cite{rozikovbook}. More generally, it is known that polygonal billiard tables with angles equal to rational multiples of $\pi$ have many periodic trajectories~\cite{singular_traj_ref, demarco2011, masur2002, sergebook}. However, almost nothing is known for billiard tables with irrational angles. For example, as we already alluded to in the introduction, it is an open question whether periodic orbits exist for triangular billiard tables with arbitrary (irrational) angles. The (rational) angles for acute triangles that lead to a \emph{lattice polygon}, that is, triangles that unfold to a planar tiling, have all been listed~\cite{smillie2002, puchta2001}. Hooper~\cite{Hooper2007} has shown that periodic orbits for right-triangle billiards are never stable. In contrast, if one perturbs the right triangle such that it has an obtuse angle, then it will always have a stable periodic orbit~\cite{schwartz2006, schwartz2009}; the proof is restricted to obtuse angles that do not exceed $\frac{5 \pi}{9}$. For larger angles, we are only aware of existence results of periodic orbits for isosceles triangles~\cite{HooperSchwartz2009}. We believe that progress can be made from the study of periodic orbits in a parallelogram table, which we view as continuous deformations from corresponding periodic orbits in the rectangular table. Results in this direction are left for future work.

\section*{Acknowledgements}
HHC is grateful to Sean Gasiorek for many fruitful discussions regarding billiards. HMO thanks Moon Duchin for her directions to the more recent literature on the topic and Sergey Tabachnikov for his comments on an earlier draft of the manuscript.



\end{document}